\documentclass[a4paper,12pt]{amsart}
\usepackage[utf8]{inputenc}
\usepackage{fullpage}
\usepackage{bookmark}
\usepackage{graphicx}
\usepackage{amssymb,amsthm,amsmath,color}
\newtheorem{prop}{Proposition}[section]

\newtheorem{thm}{Theorem}
\newtheorem{lem}{Lemma}[section]
\newtheorem{claim}{Claim}[section]

\newtheorem{rem}{Remark}[section]

\newtheorem{definition}{Definition}[section]
\renewcommand{\Re}{\mathbb{R}}
\def\S{{\mathbb{S}}}
\def\BX{{\mathbf{B}[X]}}
\def\B{{\mathbf{B}}}
\def\BXX{{\mathbf{B}[X^\prime]}}
\def\DY{{\Delta Y}}

\def\DtoY{{\Delta\textrm{-to-}Y}}
\def\YtoD{{Y\textrm{-to-}\Delta}}

\newcommand{\st}{\;:\;}

\newcommand{\inter}{\mathrm{int}}
\newcommand{\bd}{\mathrm{bd}}

\newcommand{\noshow}[1]{}

\author{Sami Mezal Almohammad\address{Sami Mezal Almohammad:
Inst. of Math., Lor\'{a}nd E\"{o}tv\"{o}s Univ., Budapest,\,Hungary.\\
Dept. of Math., Faculty of Comp. Sci. and Math., Univ. of Thi-Qar, Thi-Qar, Iraq.
}
\email{sami85@cs.elte.hu, sami.mezal@utq.edu.iq}
\and
Zsolt L\'{a}ngi\address{Zsolt L\'{a}ngi:
MTA-BME Morphodynamics Research Group and Department of Geometry, Budapest
University of Technology, Budapest, Hungary
}\email{zlangi@math.bme.hu}
\and
M\'arton Nasz\'odi\address{M\'arton Nasz\'odi: Alfr\'ed R\'enyi Inst. of Math.;
MTA-ELTE Lend\"ulet Combinatorial Geometry Research Group;
Dept. of Geometry, Lor\'and E\"otv\"os University, Budapest}
\email{marton.naszodi@math.elte.hu}
}

\title{An analogue of a theorem of Steinitz for ball polyhedra in $\Re^3$}

\begin{document}

\begin{abstract}
Steinitz's theorem states that a graph $G$ is the edge-graph of a $3$-dimensional
convex polyhedron if and only if, $G$ is simple, plane and $3$-connected. We prove
an analogue of this theorem for ball polyhedra, that is, for intersections of finitely many unit
balls in $\mathbb{R}^3$.
\end{abstract}

\keywords{Steinitz's theorem, polyhedron, ball polyhedron, edge-graph.}
\subjclass[2010]{52B10, 52A30, 52B05}
\maketitle

\section{Introduction}
Our work takes place in Euclidean $3$-space. For the closed ball of radius $\rho$
centered at $x\in \Re^3$, we use the notation $\B[x,
\rho]:=\{y\in\Re^3\st d(x,y)\leq \rho\}$. The 2-dimensional sphere (the
boundary of a closed ball) is denoted by $\S^2(x, \rho):=\{y\in\Re^3\st
d(x,y)=\rho\}$. For brevity, we set $\B[x]:=\B[x, 1]$, $\S(x):=\S^{2}(x, 1)$ and
for a set $X\subseteq\mathbb{R}^3$, we write $\BX:=\bigcap\limits_{x\in X}\B[x]$.

Let $X\subset\mathbb{R}^3$ be a finite, nonempty set contained in a ball of
radius less than 1. The set $P=\BX$ is called a \emph{ball polyhedron}. For any
$x\in X$, we call $\B[x]$ a \emph{generating ball} of $P$ and $\S(x)$ a
\emph{generating sphere} of $P$. Unless we state otherwise, we will assume that
$X$ is a \emph{reduced set of centers}, that is, that
$\BX\neq\mathbf{B}[X\setminus\{x\}]$ for any $x\in X$.

The face structure of a $3$-dimensional ball polyhedron $\BX$ are defined in a natural way: a point on the boundary of $\BX$ belonging to at least three generating spheres is called a \emph{vertex}; a connected component of the intersection of two generating spheres and $\BX$ is called an \emph{edge}, if it is a non-degenerate circular arc; and the intersection of a generating sphere and $\BX$ is called a \emph{face}.

The face structure of a ball polyhedron, unlike that of a convex polyhedron, is
not necessarily an algebraic lattice, with respect to containment, see
\cite{BN06}. Following \cite{BLNP07}, we call a ball polyhedron in
$\mathbb{R}^3$ \emph{standard}\label{page:standard}, if its vertex-edge-face
structure is a lattice with respect to containment. This is the case if, and
only if, the intersection of any two faces is either empty, or one vertex or one
edge, and any two edges share at most one vertex. The paper \cite{KMP10} and Chapter~6 of the beautiful book \cite{MMO19} by Martini, Montejano and Oliveros provide further background on the theory of ball polyhedra.

A fundamental result of Steinitz (see, \cite{Ziegler95}, \cite{Steinitz34} and
\cite{Steinitz22}) states that a graph $G$ is the edge-graph of a 3-dimensional
convex polyhedron if and only if, $G$ is simple (ie., it contains no loops and
no parallel edges), plane and \emph{$3$-connected} (ie., removing any two
vertices and the edges adjacent to them yields a connected graph). In
\cite{BLNP07}, it is shown that the edge-graph of any standard ball polyhedron
in $\mathbb{R}^3$ is simple, plane and $3$-connected. Solving an open problem
posed in \cite{BLNP07} and \cite{Bezdek13}, our main result shows that the
converse holds as well.

\begin{thm}\label{thm:mainresult}
Every $3$-connected, simple plane graph is the edge-graph of a standard
ball polyhedron in $\mathbb{R}^3$.
\end{thm}

The proof of Steinitz's theorem consists of two parts. First, it is shown that
$3$-connected, simple plane graphs can be ``reduced'' by a finite sequence of
certain graph operations to the complete graph $K_4$ on four vertices. Second,
in the geometric part, it is shown that if a graph $G$ is obtained from another
graph $G^\prime$ by such an operation and $G$ is realizable as the edge-graph of
a polyhedron, then $G^\prime$ is realizable as well. To prove
Theorem~\ref{thm:mainresult}, we use the first, combinatorial part without modification. Our
contribution is the proof of the second, geometric part in the setting of ball polyhedra.

The structure of the paper is the following. First, in Section~\ref{sec:prelim},
we introduce these operations on graphs, and recall facts on the face structure
of the dual of a ball polyhedron.
In Section~\ref{sec:proofmain}, we state our main contribution, Theorem~\ref{thm:maingeometry}, which shows the ``backward
inheritance'' of realizability by ball polyhedra under these graph operations, and deduce Theorem~\ref{thm:mainresult} from it. Finally, in Section~\ref{sec:proofgeo}, we prove Theorem~\ref{thm:maingeometry}.

\section{Preliminaries}\label{sec:prelim}

\subsection{Simple \texorpdfstring{$\DtoY$}{Delta-Y} and
\texorpdfstring{$\YtoD$}{Y-Delta} reductions on a plane graph}

Let $G$ be a $3$-connected plane graph and $K_3$ be a triangular face with vertices $v_1, v_2$ and $v_3$ (resp.,
$K_{1,3}$ be a subgraph consisting of a $3$-valent vertex $v$ of $G$, its neighbors $v_1,v_2, v_3$, and the edges $\{ v, v_i \}$ connecting $v$ to its neighbors). A \emph{$\DY$
operation} is defined as the graph operation which removes the edges $\{v_i,v_j \}$ of a triangular face $K_3$, adds a new vertex $v$ from the face, and connects it to $v_i$s, or vice versa, it takes a subgraph $K_{1,3}$ of $G$, removes the vertex $v$ and the edges incident to it, then connects all pairs $v-I, v_j$ by an edge. To specify the direction of the transformation, we will distinguish between a
$\DtoY$ transformation and a $\YtoD$ transformation (see Figure~\ref{fig:fig2}).

\begin{figure}[ht]
    \centering
 \includegraphics[width=7cm,height=3.5cm]{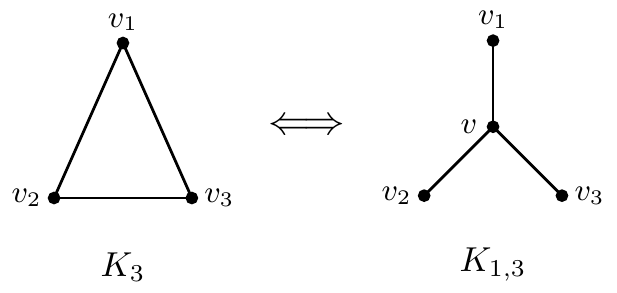}
    \caption{$\Longrightarrow$: A $\DtoY$ transformation; $\Longleftarrow$: A $\YtoD$ transformation}
    \label{fig:fig2}
\end{figure}

A $\DY$ operation may create multiple edges or vertices of degree two. A graph with such objects
is clearly not the edge-graph of a standard ball polyhedron. To fix these issues, we
define the following notion. A \emph{series-parallel reduction}, or
\emph{SP-reduction} is the replacement of a pair of edges incident to a vertex
of degree 2 with a single edge or, the replacement of a pair of parallel edges
with a single edge that connects their common endpoints, see
Figure~\ref{fig:fig1}.

\begin{figure}[ht]
    \centering
 \includegraphics[width=10cm,height=2cm]{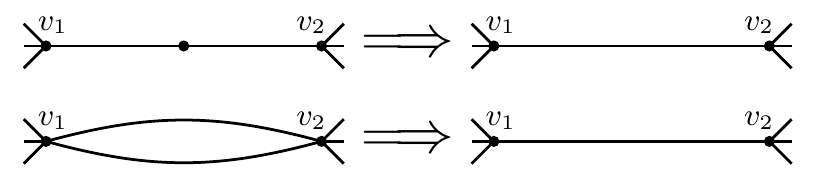}
    \caption{Examples of SP-reductions}
    \label{fig:fig1}
\end{figure}

Assume that a graph $G$ contains $K_{1,3}$ as a subgraph whose degree 3 vertex is denoted by $v$, and its neighbors are
$v_1, v_2, v_3$ (resp., $K_3$ with vertices
$v_1, v_2, v_3$), see Figure~\ref{fig:fig2}.
We call edges of $G$ that connect two vertices of $K_{1,3}$ (resp., $K_3$) \emph{internal edges}. We define
the \emph{outer degree} of a neighbor of $v$ (resp., a vertex of $K_3$),
as the number of non-internal edges adjacent to it. A $K_{1,3}$ is called $Y_0$, $Y_1$, $Y_2$,
or $Y_3$ if it has zero, one, two, or three internal edges respectively. A $K_3$ is called $\Delta_0$, $\Delta_1$,
$\Delta_2$, or $\Delta_3$ if it has zero, one, two, or three vertices of outer
degree one, respectively.

 A \emph{simple} $\DY$ reduction means any $\DY$ operation followed immediately
by SP-reductions that are then possible. There are four different types of
simple $\DtoY$ and $\YtoD$ reductions (cf. Corollary 4.7 of \cite{Ziegler95}), as shown in Figure~\ref{fig:fig3}.

\begin{prop}\label{prop:zieglercorollary}
Every $3$-connected plane graph $G$ can be reduced to $K_4$ by a sequence of
simple $\DY$ reductions.
\end{prop}

\begin{figure}[ht]
    \centering
 \includegraphics[width=13cm,height=10cm]{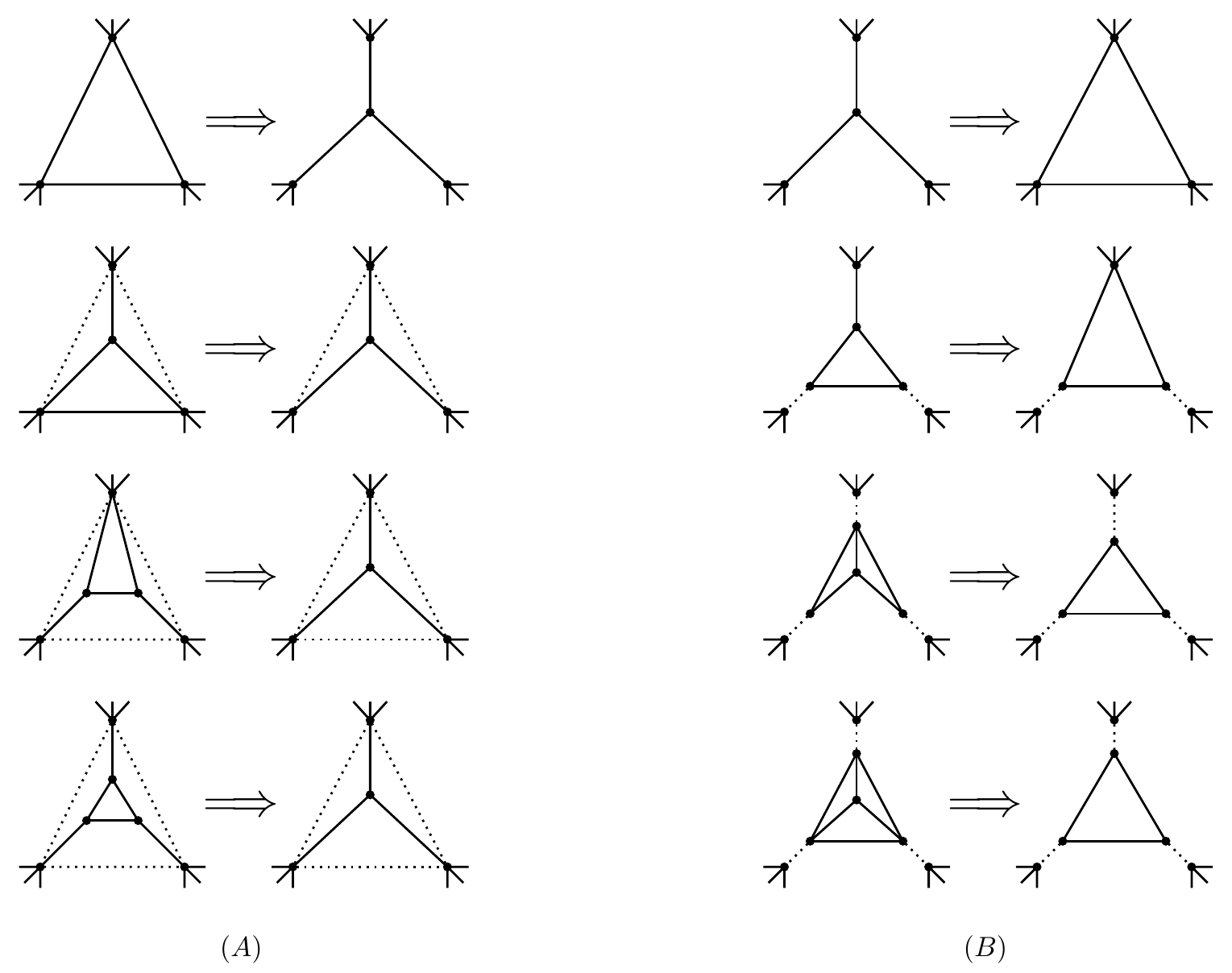}
    \caption{($A$) Four types of simple $\DtoY$ reduction, and ($B$) four types of simple $\YtoD$ reduction, where the dotted lines denote edges that may or may not be present, and are not affected by the simple $\DtoY$ and $\YtoD$ reductions.}
    \label{fig:fig3}
\end{figure}

\subsection{Standard graphs}

A planar graph with a fixed drawing on the plane is called a \emph{plane graph}. It is well known that $3$-connected planar graphs have only one drawing, that is, all plane drawings of such a graph have isomorphic face lattices \cite[Section~4.1]{Ziegler95}.

\begin{definition}
Let $G$ be a plane graph. We call $G$ \emph{standard}, if
\begin{itemize}
\item[(i)] the intersection of any two faces is either empty, or one vertex or one edge, and
\item[(ii)] any two edges share at most one vertex.
\end{itemize}
\end{definition}
\begin{rem}\label{rem:standard}
Let $G$ be the edge-graph of a ball polyhedron. Then $G$ is standard if and only
if the ball polyhedron is a standard ball polyhedron.
\end{rem}

We leave the proof of the following two lemmas to the reader as an exercise.

\begin{lem}\label{lem:dystandardgraph} Let $G$ be a $3$-connected plane graph and let the graph
$G_0$ be derived from $G$ by a simple $\DtoY$ reduction. If $G$ is a standard
graph, then so is $G_0$.
\end{lem}

The \emph{subdivision} of an edge
$\{t_1, t_2\}$ of a graph $G$ is another graph obtained from $G$ by removing
the edge $\{t_1, t_2\}$, then adding a new vertex $t^\prime$ and, finally, adding the edges
$\{t_1, t^\prime\}$ and $\{t^\prime, t_2\}$.

\begin{lem}\label{lem:standardgraphs}
Let $G$ be a standard graph, $E$ be a face of $G$ and $\{u_1, u_3\}$, $\{u_3, u_2\}$ be two edges of $E$ such that $u_1$ and $u_2$ are non-adjacent vertices.
\begin{enumerate}
    \item[\emph{I.}] If the graph $H$ is obtained from $G$ by adding the edge $\{u_1, u_2\}$, then $H$ is a standard graph.
    \item[\emph{II.}] If the graph $H^\prime$ is obtained from $G$ by adding the edge $\{u_2, u^\prime\}$ where $u^\prime$ is a new vertex subdividing the edge $\{u_1, u_3\}$, then $H^\prime$ is a standard graph.
    \item[\emph{III.}] If the graph $H^{\prime\prime}$ is obtained from $G$ by adding the edge $\{u^\prime, u^{\prime\prime}\}$ where $u^\prime$ and $u^{\prime\prime}$ are two new vertices subdividing the edges $\{u_1, u_3\}$ and $\{u_3, u_2\}$ respectively, then $H^{\prime\prime}$ is a standard graph.
\end{enumerate}
\end{lem}

\subsection{Graph duality}
We denote the dual of a plane graph $G$ by $G^\star$, see
\cite[Section~4.1]{Ziegler95}. It is well known that $G^\star$ is also a plane
graph, and $G^\star$ is $3$-connected if and only if, $G$ is 3-connected.

According to the following fact, simple $\DtoY$ reductions and simple $\YtoD$ reductions are dual to each other, see \cite[Section~4.2]{Ziegler95}.

\begin{prop}\label{prop:dyduality}
Let $G$ and $G^\prime$ be 3-connected plane graphs. Then $G^\prime$ is obtained
from $G$ by a simple $\DtoY$ reduction if and only if, $G^{\prime\star}$ is
obtained from $G^\star$ by a simple $\YtoD$ reduction.
\end{prop}

\subsection{The dual of a ball polyhedron}

In the following, $\mathcal{F}(\mathbf{B}[X])$ denotes the set of faces, and
$\mathcal{V}(\mathbf{B}[X])$ denotes the set of vertices of the ball polyhedron
$\BX$.

Let $\BX$ be a ball polyhedron in $\Re^3$ all of whose faces contain at least
three vertices. In \cite{BN06}, the \emph{dual} of $\BX$ is introduced as the ball polyhedron
$\B[\mathcal{V}(\BX)]$, and  a bijection, called the \emph{duality mapping} between $\BX$ and $\B[\mathcal{V}(\BX)]$, is given
between the faces, edges and vertices of $\BX$ and $\B[\mathcal{V}(\BX)]$, consisting of the following three mappings:

\begin{enumerate}
  \item The \emph{vertex-face} mapping is\\
  $\mathcal{V}(\BX)\ni v\mapsto V \in \mathcal{F}(\mathbf{B}[\mathcal{V}(\BX)])$\\
  where $V$ is the face of $\mathbf{B}[\mathcal{V}(\BX)]$ with $v$ as its center.
  \item The \emph{face-vertex} mapping is\\
  $\mathcal{F}(\BX)\ni F\mapsto f \in \mathcal{V}(\mathbf{B}[\mathcal{V}(\BX)])$\\
  where $f$ is the center of the sphere supporting the face $F$.
  \item The \emph{edge-edge} mapping is the following.
  Two vertices in $\mathbf{B}[\mathcal{V}(\BX)]$ are connected by an edge if and only if, the corresponding faces of $\BX$ meet in an edge.
\end{enumerate}

Note that every face of a standard ball polyhedron contains at least three
edges. The relationship between graph duality and duality of ball polyhedra is
described below.

\begin{lem}[Theorem 6.6.5., \cite{Bezdek13}]\label{lem:bezdektheorem}
Let $P$ be a standard ball polyhedron of $\mathbb{R}^3$. Then the intersection
$P^\star$ of the closed unit balls centered at the vertices of $P$ is another
standard ball polyhedron whose face lattice is dual to that of $P$.
\end{lem}

\section{Proof of Theorem \ref{thm:mainresult}}\label{sec:proofmain}

Our main contribution follows.
\begin{thm}\label{thm:maingeometry}
Let $G^\prime$ be a 3-connected plane graph, and let the graph $G$ be derived
from $G^\prime$ by a simple $\YtoD$ reduction. If $G$ is the edge-graph of a
standard ball polyhedron in $\mathbb{R}^3$, then so is $G^\prime$.
\end{thm}

First, we show how Theorem~\ref{thm:maingeometry} implies
Theorem~\ref{thm:mainresult}.

\begin{proof}[Proof of Theorem~\ref{thm:mainresult}]
Let $G$ be a $3$-connected simple plane graph. By
Proposition~\ref{prop:zieglercorollary}, the graph $G$ reduces to $K_4$ the
edge-graph of the standard ball tetrahedron by a sequence of simple $\DY$
reductions.

Now we show that the standard ball tetrahedron can be gradually turned into a
realization of $G$. Let $H$ be the edge-graph of a standard ball polyhedron and
assume that $H$ is obtained from another edge-graph $H^\prime$ by a simple $\DY$
reduction. We want to show that $H^\prime$ is realized by a standard
ball polyhedron. So we need
to discuss two cases:

\emph{First}, assume that $H$ is obtained from $H^\prime$ by a simple $\YtoD$
reduction. Then by Theorem~\ref{thm:maingeometry}, $H^\prime$ is realized by a
standard ball polyhedron.

\emph{Second}, assume that $H$ is obtained from $H^\prime$ by a simple $\DtoY$
reduction. Then by Proposition~\ref{prop:dyduality}, we get that the edge-graph
$H^\star$ is obtained from the edge-graph $H^{\prime\star}$ by a simple $\YtoD$
reduction. By Lemma~\ref{lem:bezdektheorem}, the edge-graph $H^\star$ is realized by a
standard ball polyhedron, and by Theorem~\ref{thm:maingeometry}, the edge-graph
$H^{\prime\star}$ is realized by a standard ball polyhedron. Again by
Lemma~\ref{lem:bezdektheorem}, the edge-graph $H^\prime$ is realized by a standard
ball polyhedron, and this completes the proof.
\end{proof}

\begin{center}
 \includegraphics[width=13cm,height=5cm]{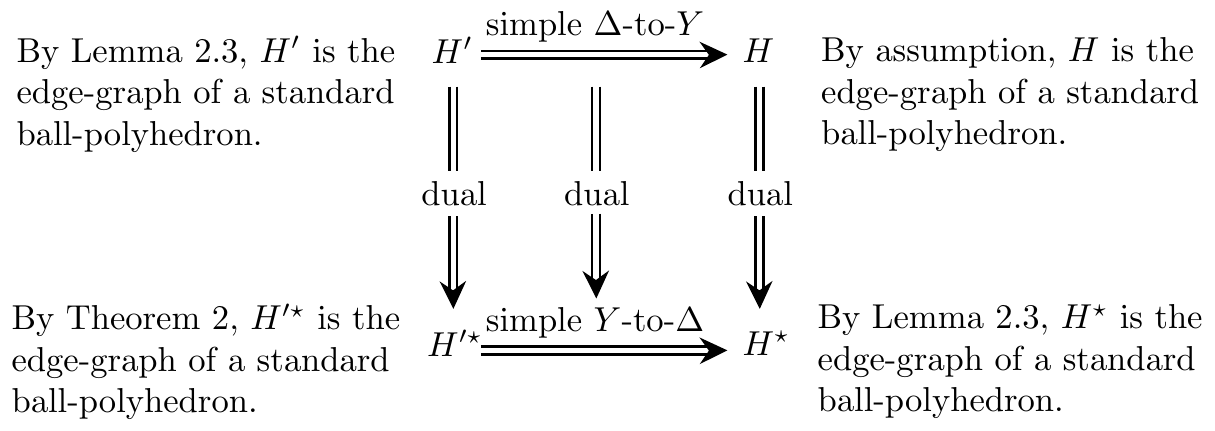}
\end{center}

\section{Proof of Theorem~\ref{thm:maingeometry}}\label{sec:proofgeo}

Let $\emptyset\neq X\subset \mathbb{R}^3$ be a finite set and $\BX$ be a
standard ball polyhedron with edge-graph $G$, and assume that $G$ is
obtained from a graph $G^\prime$ by a simple $\YtoD$ reduction. We
need to show that $G^\prime$ is realized by a standard ball polyhedron.

Let $\Lambda$ denote the triangular face of $\BX$ which realizes the triangle
obtained in the $\YtoD$ reduction, let $\S(x_\Lambda)$ be its supporting unit sphere,
and $v_1$, $v_2$, $v_3$ be the vertices of $\Lambda$ and $e_1$, $e_2$, $e_3$ the
edges. Let $F_1$, $F_2$ and $F_3$ denote the faces of $\BX$ distinct from
$\Lambda$ containing $e_1$, $e_2$ and $e_3$ respectively, and let $\S(x_1)$,
$\S(x_2)$ and $\S(x_3)$ be the unit spheres supporting these faces, see
Figure~\ref{fig:fig4}.

\begin{figure}[ht]
    \centering
 \includegraphics[width=4cm,height=3.5cm]{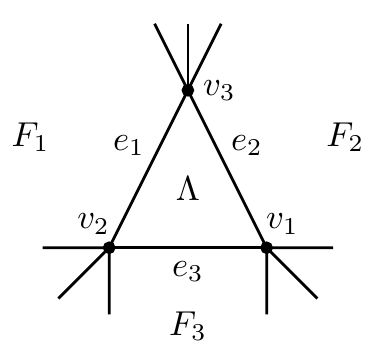}
    \caption{}
    \label{fig:fig4}
\end{figure}

The starting point of the proof of Theorem~\ref{thm:maingeometry} is the removal
of the ball that generates the triangular face $\Lambda$. Thus, we obtain another
ball polyhedron, $\mathbf{B}[X\setminus\{x_\Lambda\}]$. The following lemma
(which we prove later) describes the edge-graph of
$\mathbf{B}[X\setminus\{x_\Lambda\}]$ and, combined with Lemma~\ref{lem:dystandardgraph} yields
that it is a standard graph, and hence, by Remark~\ref{rem:standard},
$\mathbf{B}[X\setminus\{x_\Lambda\}]$ is a standard ball polyhedron.

\begin{lem}\label{lem:edgegraphofnewballpolyhedron}
The edge-graph of the ball polyhedron $\mathbf{B}[X\setminus\{x_\Lambda\}]$ is
obtained by a simple $\DtoY$ reduction applied to $\Lambda$ in the role of
$K_3$.
\end{lem}

The edge-graph of $\mathbf{B}[X\setminus\{x_\Lambda\}]$ described in Lemma~\ref{lem:edgegraphofnewballpolyhedron} may be $G^\prime$, in
which case we are done. However, it may happen that this is not $G^\prime$, more precisely, the graph $G$ is derived from $G^\prime$ by a simple $\YtoD$ reduction, but the converse is not always true, it may happen that $G^\prime$ is not derived from $G$ by a simple $\DtoY$ reduction. The reason is that when we do a $\DtoY$ reduction, the vertices of the triangle of outer degree one in $G$ become degree two vertices in the graph obtained from $G$ by a $\DtoY$ reduction. Next, we do the \emph{SP}-reduction, and these vertices are lost, see Figure~\ref{fig:fig5} (C) and (D). Moreover, the internal edges will be missing as well, see Figure~\ref{fig:fig5} (B), (C) and (D).

The following lemma describes how the edge-graph of $\mathbf{B}[X\setminus\{x_\Lambda\}]$ is converted into $G^\prime$ by adding the missing vertices and edges. We achieve this by adding some extra balls. Lemma~\ref{lem:standardgraphs} yields that $G^\prime$ is a standard graph, and hence, by Remark~\ref{rem:standard}, the ball polyhedron realizing $G^\prime$ is a standard ball polyhedron.

\begin{figure}[ht]
    \centering
 \includegraphics[width=12cm,height=12cm]{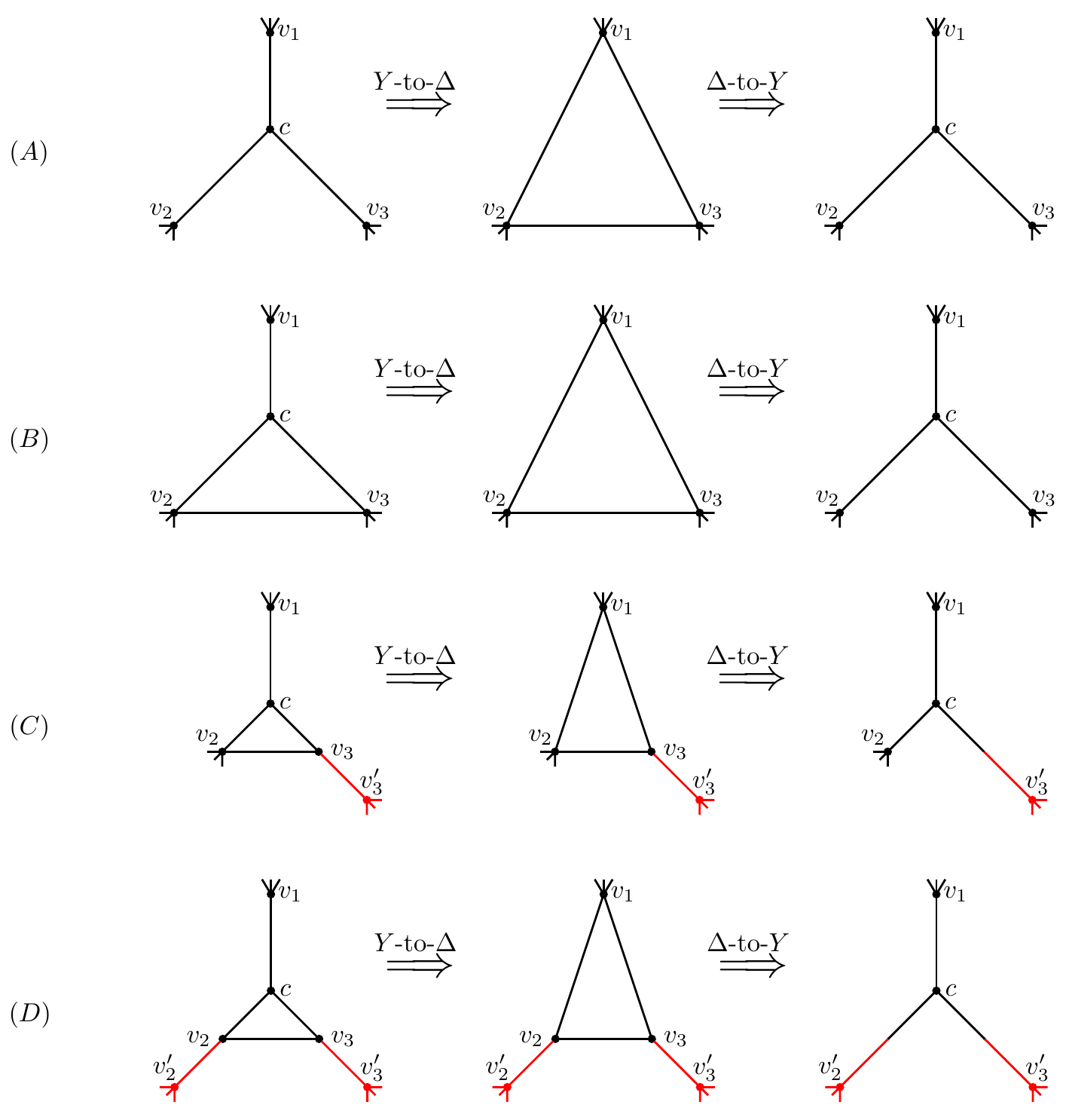}
    \caption{}
    \label{fig:fig5}
\end{figure}

\begin{lem}\label{lem:addingextraballs} Let $\emptyset\neq Z\subset\mathbb{R}^3$ be a finite set and $\B[Z]$ be a ball polyhedron. If the edge-graph of $\B[Z]$ contains $Y_0$ as an induced subgraph whose vertices are $u$, $u_1$, $u_2$ and $u_3$, and whose edges are $a_1$, $a_2$ and $a_3$, see Figure~\ref{fig:fig6}, left side, then
\begin{enumerate}
    \item[\emph{I.}] there exists a center $w$ such that the edge-graph of the ball polyhedron $\B[Z\cup \{w\}]$ is obtained from the edge-graph of $\B[Z]$ by adding the internal edge $\{u_1, u_2\}$.
    \item[\emph{II.}] there exists a center $w^\prime$ such that the edge-graph of the ball polyhedron $\B[Z\cup \{w^\prime\}]$ is obtained from the edge-graph of $\B[Z]$ by adding the edge $\{u_2, u_1^\prime\}$, where $u_1^\prime$  is a new vertex subdividing the edge $a_1$.
    \item[\emph{III.}] there exists a center $w^{\prime\prime}$ such that the edge-graph of the ball polyhedron $\B[Z\cup \{w^{\prime\prime}\}]$ is obtained from the edge-graph of $\B[Z]$ by adding the edge $\{u_1^\prime, u_2^\prime\}$, where $u_1^\prime$ and $u_2^\prime$ are two new vertices subdividing the edges $a_1$ and $a_2$ respectively.
\end{enumerate}
\end{lem}

In summary, proving Lemmas~\ref{lem:edgegraphofnewballpolyhedron} and \ref{lem:addingextraballs}, we prove Theorem~\ref{thm:maingeometry}, which in turn yields Theorem~\ref{thm:mainresult}.

\subsection{Proof of Lemma~\ref{lem:edgegraphofnewballpolyhedron}}
In this section, we use the notation of Lemma~\ref{lem:edgegraphofnewballpolyhedron} and Figure~\ref{fig:fig4}.

The following claim is obvious.
\begin{claim}\label{claim:interiorpointsoftriangularface} For any $i\in\{1, 2, 3\}$, $\Lambda\setminus e_i$ is contained in the interior of $\B[x_i]$.
\end{claim}

The following claim is the key of our proof. It states that the ``new part'' of the boundary of the new ball polyhedron $\mathbf{B}[X\setminus\{x_\Lambda\}]$ belongs to the union of $\mathbb{S}(x_1)$, $\mathbb{S}(x_2)$ and $\mathbb{S}(x_3)$.

\begin{claim}\label{claim:threespheres} $\bd(\BXX)\setminus \bd(\BX)\subseteq \mathbb{S}(x_1)\cup \mathbb{S}(x_2)\cup\mathbb{S}(x_3)$, where $X^\prime=X\setminus\{x_\Lambda\}$.
\end{claim}
\begin{proof} Consider a point $q \in \bd(\BXX)\setminus \bd(\BX)$. Then there exists a generating sphere $\mathbb{S}(x_q)$ of $\mathbf{B}[X^\prime]$ such that $q\in \mathbb{S}(x_q)$ and $x_q\in X^\prime$, implying that $\mathbb{S}(x_q)$ is a generating sphere of $\BX$ as well. Let $F=\mathbb{S}(x_q)\cap \BX$ and $F^\prime=\mathbb{S}(x_q)\cap \BXX$. Then $F=F^\prime\cap \mathbf{B}[x_{\Lambda}]$, $F\subseteq F^\prime$, $q\in F^\prime$, and $q\notin F$. This yields that $F^\prime\cap \mathbb{S}(x_{\Lambda})$ is a non-degenerate circular arc in $F^\prime$ that separates $q$ from $F$. Thus $F$ intersects $\Lambda$ in a non-degenerate circular arc. That only happens if $F$ intersects $\Lambda$ in an edge of $\mathbf{B}[X]$, and hence, $x_q= x_1, x_2$, or $x_3$.
\end{proof}

The following Claim is obvious.
\begin{claim}\label{claim:threeballs}
Let $\B_1, \B_2$ and $\B_3$ be closed unit balls in $\mathbb{R}^3$ such that $\B_1\cap \B_2\cap \B_3$ is a ball polyhedron with three faces. Then $\B_1\cap \B_2\cap \B_3$ is a ball polyedron with two vertices connected by three edges.
\end{claim}

Finally, we are in the position to prove Lemma~\ref{lem:edgegraphofnewballpolyhedron}. By Claim~\ref{claim:threeballs}, the boundary of $\B[x_1]\cap\B[x_2]\cap\B[x_3]$ contains two vertices of degree three say $q$ and $\bar{q}$, three edges, and three faces. We need to prove that the ``new part'' $\mathcal{N}:=\bd(\mathbf{B}[X\setminus\{x_\Lambda\}])\setminus \bd(\BX)$ of the boundary of the ball polyhedron $\mathbf{B}[X\setminus\{x_\Lambda\}]$ contains either $q$ or $\bar{q}$ with part from each of the three edges, and part from each of the three faces (i.e., $K_{1,3}$). It means that when we remove the ball $\B[x_\Lambda]$, the triangular face $\Lambda$ of $G$ will be replaced by $Y_0=K_{1,3}$ in the edge-graph of $\mathbf{B}[X\setminus\{x_\Lambda\}]$, i.e., the edge-graph of $\mathbf{B}[X\setminus\{x_\Lambda\}]$ is derived from $G$ by a simple $\DtoY$ reduction.

Let $\Gamma:=\B[x_1]\cap\B[x_2]\cap\B[x_3]$ and $\gamma:=\bd(\Lambda)=e_1\cup e_2\cup e_3$, see Figure~\ref{fig:fig4}. By Claim~\ref{claim:threespheres}, $\mathcal{N}\subseteq\S(x_1)\cup\S(x_2)\cup\S(x_3)$, this implies that $\mathcal{N}\subseteq \bd(\Gamma)$. Clearly, $\bd(\Gamma)$ has three edges, say $e^\prime_1, e^\prime_2$ and $e^\prime_3$ such that $e^\prime_1\subseteq\S(x_2)\cap\S(x_3)$, $e^\prime_2\subseteq\S(x_1)\cap\S(x_3)$ and $e^\prime_3\subseteq\S(x_1)\cap\S(x_2)$.

By Claim~\ref{claim:interiorpointsoftriangularface}, $v_i\in \inter\left(\B[x_i]\right)$ for all $i=1, 2, 3$. By Claim~\ref{claim:threeballs}, $\S(x_1)\cap\S(x_2)\cap\S(x_3)$ is a set of two points $q$ and $\bar{q}$. Clearly, $\{v_1, v_2, v_3\}\cap\{q,\ \bar{q}\}=\emptyset$.

Since $e^\prime_1\setminus\{q,\ \bar{q}\}\subseteq \inter(\B[x_1])$, $v_1\in e^\prime_1\setminus\{q,\ \bar{q}\}$ and $\{e^\prime
_2, e^\prime_3\}\subseteq\S(x_1)$, this implies that $v_1$ belongs to $e^\prime_1\setminus\{q,\ \bar{q}\}$ and does not belong to $e^\prime_2$ or to $e^\prime_3$. Similarly, $v_i$ belongs to $e^\prime_i\setminus\{q,\ \bar{q}\}$ ($i=2, 3$) only, respectively. Thus, exactly one of $v_1, v_2$ and $v_3$ is contained on each of the three edges of $\Gamma$.

Observe that both $\S(x_\Lambda)\cap \Gamma$ and  $\Lambda$ are 
the intersections of three spherical disks on $\S(x_\Lambda)$, each smaller than a hemi-sphere: $\S(x_\Lambda)\cap \B[x_1]$, $\S(x_\Lambda)\cap \B[x_2]$ and $\S(x_\Lambda)\cap \B[x_3]$. Hence, $\S(x_\Lambda)\cap \Gamma=\Lambda$, and it follows that $\S(x_\Lambda)\cap \bd(\Gamma)=\gamma$.

It follows that $\gamma$ partitions $\bd(\Gamma)$ into two components, $q$ is in one component and $\bar{q}$ is in the other, we may assume that $\bar{q}\in\inter\left(\B[x_\Lambda]\right)$. Claim~\ref{claim:threespheres} yields that $q\in\mathcal{N}$ as required. This completes the proof of Lemma~\ref{lem:edgegraphofnewballpolyhedron}.

\subsection{Proof of Lemma~\ref{lem:addingextraballs}}\;

(I) Let $g_1$, $g_2$ and $g_3$ be the faces of $\B[Z]$ such that $a_1=g_2\cap g_3$, $a_2=g_1\cap g_3$  and $a_3=g_1\cap g_2$ and let $\S(z_1)$, $\S(z_2)$ and $\S(z_3)$ be the spheres supporting these faces, see Figure~\ref{fig:fig6}, left side.

To add an internal edge to $Y_0$, we will add a rotated copy $z_3^\prime$ of $z_3$ to the set $Z$, where the axis of the rotation is the line through $u_1$ and $u_2$, the angle of the rotation is sufficiently small, and $u$ is outside of $\B[z_3^\prime]$. Thus, we obtain a new triangular face $g_3^\prime$ supported by $\S(z_3^\prime)$, see the dashed lines on Figure~\ref{fig:fig6} ($A$), and remove the dotted lines on Figure~\ref{fig:fig6} ($A$).

\begin{figure}[ht]
    \centering
 \includegraphics[width=10cm,height=10cm]{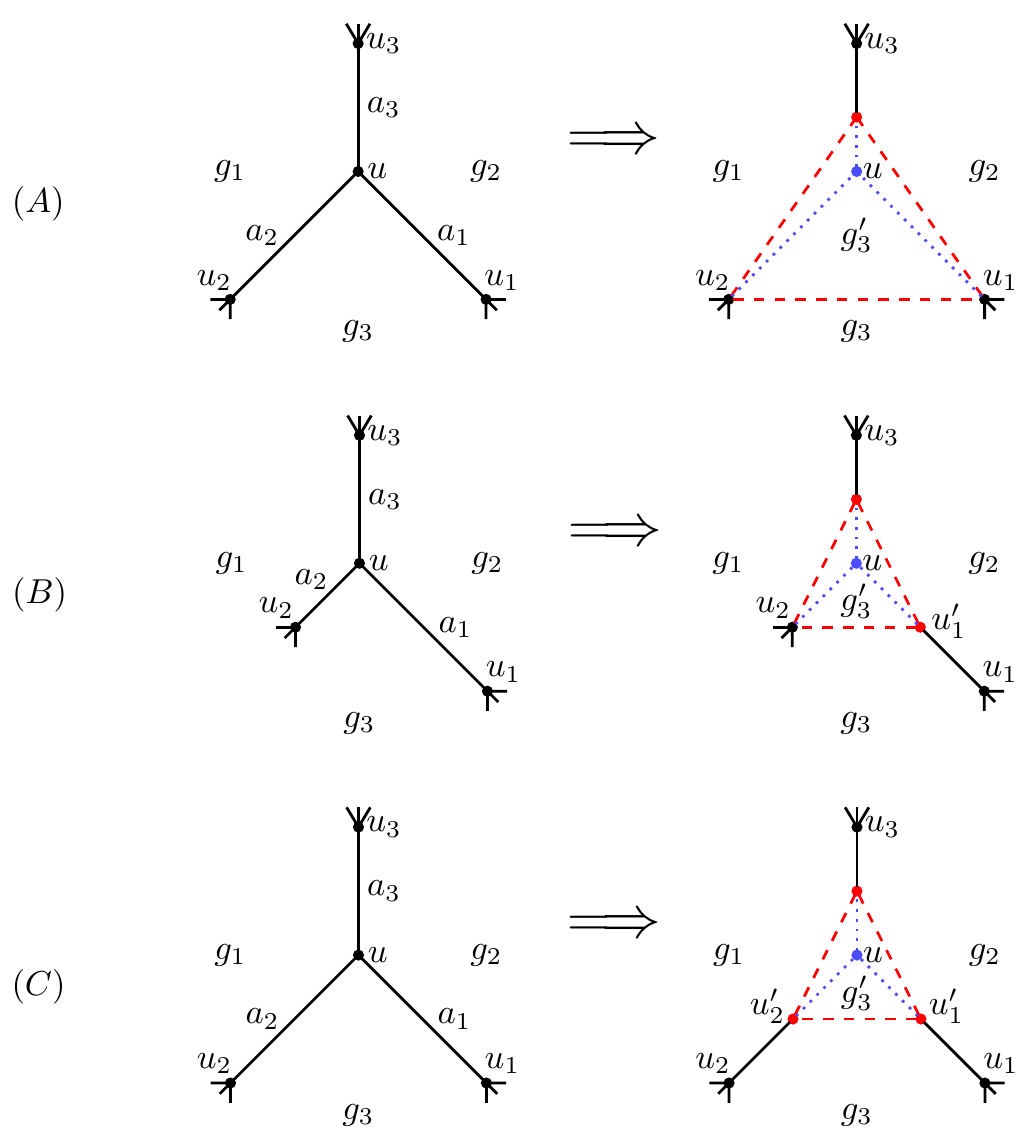}
   \caption{The dotted (blue) vertices and edges are removed and the dashed (red) vertices and edges are introduced in the new graph.}
    \label{fig:fig6}
\end{figure}

(II) To add a new vertex and a new edge to $Y_0$, we use the same method as in (I), but we choose the rotation axis so that it passes through the vertex $u_2$ and intersects the edge $a_1$ at a point, say $u_1^\prime$, distinct from its endpoints, see Figure~\ref{fig:fig6} ($B$).

(III) To add two new vertices and a new edge to $Y_0$, we use again the same method of (I), but this time, we choose the rotation axis so that it intersects the edges $a_1$ and $a_2$ at two points, say $u_1^\prime$ and $u_2^\prime$ respectively,
distinct from their endpoints, see Figure~\ref{fig:fig6} ($C$). This finishes the proof of Lemma~\ref{lem:addingextraballs}.

\section*{Acknowledgements}

SMA would like to thank the Tempus Public Foundation (TPF), Stipendium Hungaricum program, and University of Thi-Qar, Iraq for the support for his PhD scholarship.

ZL was supported by grants K119670 and BME Water Sciences \& Disaster Prevention TKP2020 IE of the National Research, Development and Innovation Fund (NRDI), by the \'UNKP-20-5 New National Excellence Program of the Ministry for Innovation and Technology, and the J\'anos Bolyai Scholarship of the Hungarian Academy of Sciences.

MN was supported by the National Research, Development and Innovation Fund (NRDI) grant K119670, by the \'UNKP-20-5 New National Excellence Program of the Ministry for Innovation and
Technology from the source of the NRDI, as well as the J\'anos Bolyai Scholarship of the Hungarian Academy of Sciences.

\bibliographystyle{amsalpha}
\bibliography{biblio}

\end{document}